\documentclass[a4paper]{article}
\usepackage[english]{babel}
\usepackage[cp1251]{inputenc}
\usepackage[colorlinks,linkcolor=blue,citecolor=blue,%
urlcolor=blue,filecolor=blue,pagecolor=blue]{hyperref}
\usepackage{latexsym,amsfonts,amssymb,amsthm,amsmath,amsbsy,graphicx}
\pdfoutput=1
\newtheorem{thm}{Theorem}
\newtheorem{lem}{Lemma}
\newtheorem{nas}{Corollary}
\begin{document}
\thispagestyle{myheadings} \markboth{Ukrainian Mathematical Journal,
Vol. 62, No. 1, 2010}{Ukrainian Mathematical Journal,
Vol. 62, No. 1, 2010}
\bigskip
 {\noindent \bf \large  BEHAVIOR OF AN ALMOST SEMICONTINUOUS POISSON PROCESS
ON A MARKOV CHAIN UPON ATTAINMENT OF A LEVEL}\footnotemark[1]
\footnotetext{Translated from
Ukrains’kyi Matematychnyi Zhurnal,
Vol. 62, No. 1, pp. 81–89, January,
2010.  Original article submit\-ted April 4, 2008;
revision submitted February 19, 2009.
 This reprint differs from the original in
pagination and typographic detail.}

\bigskip
\bigskip
 { \bf Ievgen~Karnaukh
 \footnotetext{Kyiv National University of Trade
 and Economics, Kyiv, Ukraine.\phantom{\;}
 \href{mailto:ievgen.karnaukh@gmail.com}{ievgen.karnaukh@gmail.com}
 }}\hskip 10 cm UDC 519.21
\begin{center}
\bigskip
\begin{quotation}
\noindent  {\small We consider the almost
semi-continuous processes defined on a finite Markov chain.
The representation of the moment generating functions
for the absolute maximum after achievement positive level and for
the recovery time are obtained. Modified processes with two-step
rate of negative jumps are investigated.}
\end{quotation}\footnotetext{\bigskip \hskip 2 cm 0041–5995/10/6201–0087 \;\;
\copyright\; 2010\;\; Springer Science+Business Media, Inc.}
\end{center}

\section{Introduction}
In the present paper, we continue the investigations of almost
 semicontinuous processes defined on finite
Markov chains originated in~\cite{KarnaukhUMZ,KarnaukhTBIMS}.
In the first section, we consider analogs of the functionals
 studied in~\cite{Gusak2007}
 (Sec. 6.3) for the scalar case.
 In the second section, we study overjump functionals for a
 modified almost semi-
continuous process playing the role of an analog of a modified
 semicontinuous process with drift whose variations
depend on the attained level (see, e.g.,~\cite{Asmussen},
 Chap. VII and~\cite{Jasiulewicz}; for the scalar case,
 see~\cite{Bratiychuk}).

 Let $x(t)$ be a finite irreducible Markov chain with the set of
  states
  $E'=\{1, \ldots, m\}$ and an infinitesimal matrix $\mathbf{Q}$.
  A process $\xi (t)$ is defined as follows: $\xi (0)=0$; for
 $x(t)=k$, $k=1,...,m$,
the increments $\xi (t)$ coincide with the increments of the process
$$
 \xi_k (t)=\sum_{n\leq \varepsilon  _k(t)}\xi^k _n-
 \sum_{n\leq \varepsilon'_k(t)}{\xi '_n} ^k
,
$$
where $\varepsilon'_k(t)$ and $\varepsilon_k(t)$ are
Poisson processes with the rates $\lambda _k^1$ and $\lambda _k^2$,
respectively. ${\xi
'_n} ^k$ and $\xi_n ^k$ are independent positive random variables. Moreover,
 ${\xi
'_n} ^k$ have exponential distributions with parameters $c_k$, whereas $\xi_n
^k$ have absolutely continuous distributions with finite expectations $m_k$.
In this case, $Z(t)=\left\{\xi
(t),x(t)\right\}$ is an almost lower semicontinuous process on a Markov chain
(see~\cite[p. 562]{KarnaukhUMZ}) with cumulant
\begin{equation}\label{eq1}
\boldsymbol{\Psi }(\alpha )=\boldsymbol{\Lambda
}{\mathbf{F}}_0(0)\left(\mathbf{C}
  \left(\mathbf{C}+\imath\alpha \mathbf{I}\right)^{-1}-\mathbf{I} \right)
  +\int_0^{\infty}\left(e^{\imath\alpha x}-\mathbf{I} \right)\boldsymbol{\Pi }(dx)
  +\mathbf{Q},
\end{equation}
where $\boldsymbol{\Lambda }=\|\delta _{kr}(\lambda _k^1+\lambda _k^2
)\|$, $\mathbf{C}=\|\delta _{kr}c_k\|$, ${\mathbf{F}}_0(0)=\|\delta
_{kr}\lambda _k^1/(\lambda _k^1+\lambda _k^2 )\|$,
$\boldsymbol{\Pi}(dx)=\boldsymbol{\Lambda
}\overline{\mathbf{F}}_0(0)d\mathbf{F}^1_0(x)$,
$\overline{\mathbf{F}}_0(0)=\mathbf{I}-\overline{\mathbf{F}}_0(0)$ and
$\mathbf{F}^1_0(x)=\|\delta _{kr}\mathrm{P}\left\{\xi
^k_n<x\right\}\|$, $x<0$.

By
$$
\xi ^{\pm }(t)=\sup\limits_{0\leq u\leq t}(\inf)\xi (u),  \xi ^{\pm
}=\sup
   \limits_{0\leq u\leq \infty}(\inf)\xi (u),
 \overline{\xi }(t)=\xi (t)-\xi ^+(t),
$$
we denote the extrema of the process $\xi(t)$. The
overjump functionals are specified as follows:
$$\tau^+(x)=\inf\{t:\xi (t)>x\}, \gamma ^+(x)=\xi (\tau
^+(x))-x, \gamma _+(x)=x-\xi (\tau ^+(x)-0),
x\geq 0; $$
$$ \tau^-(x)=\inf\{t:\xi (t)<x\} ,  x\leq 0; \tau ^-(x)=0, x>0.
$$

Let $\theta _s$ be an exponentially distributed random variable
with parameter
$s>0$ independent of $Z(t)$. The distributions of extrema and the
corresponding atomic probabilities are defined as follows:
 $$ \mathbf{P}_\pm(s,x)=\left \| \mathrm{P}\left\{\xi
^\pm(\theta_s)<x, x(\theta_s)=r/x(0)=k\right\} \right \|
=\mathbf{P}\left\{\xi ^\pm(\theta_s)<x\right\}, x>0;$$
$$
\mathbf{P}^-(s,x)=\mathbf{P}\left\{\overline{\xi}(\theta_s)<x\right\},
\;x<0;$$
 $\mathbf{p}_{\pm}(s)=\mathbf{P}\left\{\xi
^{\pm}(\theta_s)=0\right\}\mathbf{P}^{-1}_s,$
$\mathbf{P}_s=s\left(s\mathbf{I}-\mathbf{Q}\right)^{-1}$, $
\mathbf{p}^-(s)=\mathbf{P}^{-1}_s\mathbf{P}\left\{\overline{\xi}
(\theta_s)=0\right\}$,  and $\mathbf{q}^-(s)=\mathbf{I}-\mathbf{p}^-(s)$.

Further, we denote $\mathbf{R}_-(s)=\mathbf{C}\mathbf{p}_-(s)$ and
$\mathbf{R}^-(s)=\mathbf{p}^-(s)\mathbf{C}$.
Thus, for $x\leq 0$, we can write (see~\cite{KarnaukhTBIMS})
\begin{equation}\label{eq2}\begin{aligned}
\mathbf{P}_{-}(s,x) =\mathbf{E}\,\left[e^{-s\tau ^-(x)},\tau
^-(x)<\infty \right]\mathbf{P}_s= \mathbf{q}_-(s)e^{\mathbf{R}_-(s)
x}\mathbf{P}_s,\\
 \mathbf{P}^{-}(s,x) =\mathbf{P}_s
e^{\mathbf{R}^-(s) x}\mathbf{q}^-(s).\quad \quad \quad \quad \quad
\quad
\end{aligned}
\end{equation}
Since
$$\lim_{x\rightarrow
-\infty}\mathbf{P}^-(s,x)=\lim_{x\rightarrow
-\infty}\mathbf{P}_-(s,x)=0,$$
we  conclude that the spectra of the matrices
$\mathbf{R}_-(s)$ and $\mathbf{R}^-(s)$ $\left(\sigma \left(
\mathbf{R}_-(s)\right) \text{ and } \sigma\left(
\mathbf{R}^-(s)\right) \right)$  are formed by positive elements.

The following assertion for overjump functionals
is obtained from~\cite[p. 48]{KarnaukhTBIMS}:
\begin{lem} For the process $Z(t)$ with cumulant~\eqref{eq1},
\begin{multline}\label{eq3}
 \mathbf{f}_s(dx,dy/u)=\mathbf{E}\!\left[e^{-s\tau ^+(u)},
 \gamma _+(u)\in dx, \gamma ^+(u)\in dy, \tau
 ^+(u)<\infty\right]\!=\\
\shoveleft\qquad\qquad \quad  =s^{-1}d_x\mathbf{P}_+(s,u-x)
\mathbf{p}^-(s)\boldsymbol{\Pi}(dy+x)I\{x<u\}+\\
  +s^{-1}\int\limits_{0\vee(u-x)}^{u}
  d\mathbf{P}_+(s,z)\mathbf{R}^-(s) e^{\mathbf{R}^-(s)(u-x-z)}
  \mathbf{q}^-(s)\boldsymbol{\Pi}(dx+y)dy.
\end{multline}
\begin{multline}\label{eq4}
 \mathbf{g}_s(dy/u)=\mathbf{E}\!\left[e^{-s\tau ^+(u)},
  \gamma ^+(u)\in dy, \tau ^+(u)<\infty\right]=
  s^{-1}\int_{0}^{u}d\mathbf{P}_+(s,z)\times\\
  \times
  \biggl(\mathbf{p}^-(s)\boldsymbol{\Pi}(dy+u-z)
  +\mathbf{R}^-(s)\!\int_{u-z}^{\infty}\!\!
   e^{\mathbf{R}^-(s)(u-x-z)}
  \mathbf{q}^-(s)\boldsymbol{\Pi}(dx+y)dy\biggr).
\end{multline}
\begin{multline}\label{eq5}
 \mathbf{g}(dy/u)=\lim_{s\rightarrow 0}\mathbf{g}_s(dy/u)=\\
 =
 \int_{0}^{u}\!d\mathbf{M}_+(z)
 \!\biggl(\!\boldsymbol{\Pi}(dy+u-z)
  +\mathbf{C}\!\int_{u-z}^{\infty}\!\!
   e^{\mathbf{R}^-(0)(u-x-z)}
  \left(\mathbf{I}-\mathbf{p}^-(0)
  \right)\boldsymbol{\Pi}(dx+y)dy\biggr),
\end{multline}
where $\mathbf{M}_+(x)$: $
  \int_{0}^{\infty}e^{\imath\alpha x}d\mathbf{M}_+(x)=
  -\boldsymbol{\Psi}^{-1}(\alpha)
\left(\mathbf{C}+\imath\alpha \mathbf{I}
\right)^{-1}\left(\mathbf{C}\mathbf{p}^-(0) +i\alpha
\mathbf{I}\right).$
\end{lem}
Note that $d\mathbf{M}_+(x)=\lim\limits_{s\rightarrow
0}s^{-1}d\mathbf{P}_+(s,x)\mathbf{p}^-(s)$ and the matrices
$\mathbf{p}_-(0)$ and $\mathbf{p}^-(0)$ satisfies the equations
$$ \left(\boldsymbol{\Lambda }-
  \mathbf{Q}\right)\left(\mathbf{I}-\mathbf{p}_-(0) \right)=
  \boldsymbol{\Lambda }{\mathbf{F}}_0(0)+\int^{\infty}_{0}\boldsymbol{\Pi}(dz)
  \left(\mathbf{I}-\mathbf{p}_-(0)\right)e^{-\mathbf{C}\mathbf{p}_-(0)z},$$
  $$ \left(\mathbf{I}-\mathbf{p}^-(0) \right)\left(\boldsymbol{\Lambda }-
  \mathbf{Q}\right)=
  \boldsymbol{\Lambda }{\mathbf{F}}_0(0)+\int^{\infty}_{0}
  e^{-\mathbf{p}^-(0)\mathbf{C}z}\left(\mathbf{I}-\mathbf{p}^-(0)\right)\boldsymbol{\Pi}(dz),$$
respectively.

\section{Red period}

In the present section, we consider functionals connected with the behavior of
$\xi(t)$ upon attainment of a
positive level. Denote
$$z^+(u)=\sup_{\tau ^+(u)\leq t<\infty}\xi (t)-u,\;
\tau '(u)=\inf\{t>\tau^+ (u),\xi (t)<u\},$$
$$
T'(u)=
  \begin{cases}
    \tau'(u) -\tau^+ (u), & \tau ^+(u)<\infty, \\
    \infty, & \tau^+ (u)=\infty.
  \end{cases}$$

It is worth noting that the process  $Z(t)$  can be regarded
 as a surplus risk process with stochastic function
of premiums (the values of premiums are exponentially
distributed) in a Markov environment and the functionals
$z^+(u)$, $\tau '(u)$, $T'(u)$  can be regarded as the total
deficit after ruin, recovery time, and "red period"{}, respectively
(see~\cite{Rolsky}).
\begin{thm}   For the process $Z(t)$ with cumulant~\eqref{eq1}
\begin{equation}\label{eq6}
\mathbf{P}\left\{z^+(u)<x,\tau^ +(u)<\infty\right\}=
\int_{0}^{x}\mathbf{g}(dy/u)\mathbf{P}\left\{\xi ^+<x-y\right\}.
\end{equation}
\begin{multline}\label{eq7}
  s\mathbf{E}\left[e^{-s\tau '(u)},\tau '(u)<\infty\right]
  =\int_{0}^{u}d\mathbf{P}_+(s,x)\mathbf{p}^-(s)\biggl(
  \int_{u-x}^{\infty}\boldsymbol{\Pi}(dz)
  \mathbf{q}_-(s)
  e^{\mathbf{R}_-(s)(u-x-z)}+\\
  +\mathbf{C}\int_{-\infty}^{0}e^{\mathbf{R}^-(s) y}
  \mathbf{q}^-(s)\int_{u-x-y}^{\infty}\boldsymbol{\Pi}(dz)
  \mathbf{q}_-(s)e^{\mathbf{R}_-(s) (u-x-y-z)}dz\biggr),
\end{multline}
\begin{multline}\label{eq8}
  \mathbf{E}\!\left[e^{-sT'(u)},T'(u)<\infty\right]
=\int_{0}^{u}d\mathbf{M}_+(x)\biggl(
  \int_{0}^{\infty}\boldsymbol{\Pi}(dy+u-x)
  \mathbf{q}_-(s)e^{-\mathbf{R}_-(s) y}+\\
  +\mathbf{C}\int_{u-x}^{\infty}e^{\mathbf{R}^-(0)(u-x-z)}
  \mathbf{q}^-(0)
  \int_{0}^{\infty}\boldsymbol{\Pi}(dy+z)\mathbf{q}_-(s)e^{-\mathbf{R}_-(s) y}dz\biggr).
\end{multline}
\end{thm}
\begin{proof}  In view of the fact that, under the condition  $\gamma
^+(u)\in dy,\tau ^+(u)<\infty$, the functional is $z^+(u)$
stochastically equivalent to $y+\xi^+$,  we find
$$\mathbf{P}\left\{z^+(u)<x,\tau ^+(u)<\infty\right\}
=\int_{0}^{x}\mathbf{P}\left\{\gamma ^+(u)\in dy,\tau
^+(u)<\infty\right\}\mathbf{P}\left\{y+\xi^+<x\right\}.$$
In exactly the same way as in the proof of Theorem 5.1
 in~\cite{Gusak},
for the moment generating function of the time to recovery, we deduce
\begin{multline*}
\mathbf{E}\left [  e^{-s\tau '(u)},\tau
'(u)<\infty\right]=\int_{0}^{u}d\mathbf{P}_+(s,x)
\int_{-\infty}^{0}\mathbf{P}^{-1}_s\mathbf{P}^-(s,y)\times\\
\times \int_{u-x-y}^{\infty}\boldsymbol{\Pi}(dz)\mathbf{E}\left [
e^{-s\tau^-(u-x-y-z)},\tau ^-(u-x-y-z)<\infty\right].
\end{multline*}
Combining this relation with~\eqref{eq2},
 we obtain~\eqref{eq7}.
 By using the strict Markov property, we get
\begin{multline*}
  \mathrm{E}\,\left[ e^{-sT'(u)}, T'(u)<\infty, x(T'(u))=r/x(0)=k\right]=\\
\shoveleft  =\!\!\int_{0}^{\infty}\!\!\!\mathrm{E}\!\left[ e^{-sT'(u)},
  T'(u)\!<\infty, \gamma ^+(u)\in dy, x(T'(u))=r/x(0)=k\right]=\\
  =\sum_{j=1}^{m}\int_{0}^{\infty}\!\mathrm{E}\left[ e^{-sT'(u)},
  T'(u)<\infty,\gamma ^+(u)\in dy,
  x(T'(u))=r, x(\tau ^+(u))=j/x(0)=k\right]\\
=\sum_{j=1}^{m}\int_{0}^{\infty}
\mathrm{P}\left\{x(\tau ^+(u))=j,\gamma ^+(u)\in dy,\tau ^+(u)<\infty/x(0)=k\right\}\times \\
\times\mathrm{E}\!\!\left[ e^{-sT'(u)},
  T'(u)<\infty, x(T'(u))=r/\gamma ^+(u)\in dy,x(\tau ^+(u))=j,
  \tau ^+(u)<\infty, x(0)=k\!\right]
\end{multline*}
\begin{multline*}
={\sum_{j=1}^{m}\int_{0}^{\infty}\mathrm{E}\,\left[ e^{-s\tau
^-(-y)},
  \tau ^-(-y)<\infty, x(\tau ^-(-y))=r/x(0)=j\right]\times} \\
  \times \mathrm{P}\left\{x(\tau ^+(u))=j,\gamma ^+(x)\in dy,\tau ^+(u)<\infty/x(0)=k\right\}.
\end{multline*}
In deducing this equality, we have used the fact that, under the condition $\gamma ^+(u)\in dy,
 \tau ^+(u)<\infty$, the functional $T'(u)$ is stochastically equivalent
 to the time of attainment of the level $-y$.
 In the matrix form, we can write
\begin{multline*}
  \mathbf{E}\,\left[e^{-sT'(u)},T'(u)<\infty,\tau ^+(u)<\infty \right]=\\
  =\int_{0}^{\infty}
  \mathbf{P}\left\{\gamma ^+(u)\in dy, \tau ^+(u)<\infty\right\}
  \mathbf{E}\,\left[e^{-s\tau ^-(-y)},\tau ^-(-y)<\infty \right].
\end{multline*}
By using~\eqref{eq2} and~\eqref{eq5}, we establish equality~\eqref{eq8}.
\end{proof}

\section{Modified Process}
In the present section, in addition to the results obtained for the overjump functionals, we use the relations for
two-limit functionals. By
$$\tau (u,b)=\{t>0:\xi (t)\notin (u-b,u)\}$$
 we denote the time of exit from the interval $(u-b,u)$.
Further, we consider the events specifying the times of exit
through the upper and lower boundaries of the interval:
$$ A_+(u)=\left\{\omega : \xi (\tau
(u,b))\geq u\right\}\text{ and }  A_-(u)=\left\{\omega : \xi (\tau
(u,b))\leq u-b\right\},$$ and the corresponding overjumps
\begin{gather*} \gamma _b^+(u)=\xi(\tau (u,b))-u, \gamma
^b_+(u)=u-\xi(\tau
(u,b)-0) \text{ on } A_+(u);\\
 \gamma _b^-(u)=(u-b)-\xi(\tau (u,b)), \gamma
^b_-(u)=\xi(\tau (u,b)-0)-(u-b) \text{ on } A_-(u).
\end{gather*}
It follows from the results presented in~\cite[p.559]{KarnaukhUMZ} that
\begin{equation}\label{eq9}
\mathbf{E}\,\left [  e^{-s \tau (u,b)}, \gamma _b^-(u)\in dy,
A_-(u)\right ]
 =\mathbf{E}\,\left [  e^{-s \tau (u,b)},  A_-(u)\right
 ]\mathbf{C}e^{-\mathbf{C}y}dy=\mathbf{B}_b(s,u)\mathbf{C}e^{-\mathbf{C}y}dy,
\end{equation}
\begin{multline*} \mathbf{f}^+_{b,s}(dx,dy/u) =\mathbf{E}\,\left [
e^{-s \tau (u,b)}, \gamma
_+^b(u)\in dx, \gamma _b^+(u)\in dy, A_+(u)\right ]=\\
 =\!\biggl\|\mathrm{P}
   \bigl\{u-\xi (\theta_s)\in dx,\tau (u,b)>\theta_s, x(\theta
   _s)=r/x(0)=k\!\bigr\}\!\biggr\|\boldsymbol{\Pi}(dy+x)I\{0<x<b\}=\\
   =d_x\mathbf{H}_s(b,u,u-x)\boldsymbol{\Pi}(dy+x)I\{0<x<b\}.
\end{multline*}
Note that the representations for $\mathbf{B}_b(s,u)$  and
$d_x\mathbf{H}_s(b,u,x)$ were obtained in~\cite{KarnaukhUMZ}
 for almost upper semicontinuous processes.
 In this case, one can use the fact that if $\{\xi(t),x(t)\}$
 is an almost upper semicontinuous process, then $\{-\xi(t),x(t)\}$
is an almost lower semicontinuous process.

\par We now determine the modified process $\xi_{a,b}(t)$, $0<a\leq
b<\infty$.  Assume that the rates of exponentially
distributed negative jumps
$\xi_{a,b}(t)$ depend on the threshold levels $a$
and $b$ (see~\cite{Bratiychuk}). In the risk theory, this
process has the following interpretation: As soon as the reserve
 of an insurance company attains a certain level,
the company may decrease the value of the premium to attract
 additional clients.  Therefore, the distribution of
 the values of premiums contains a parameter
$\tilde{\mathbf{C}}=\mathbf{C}(r)$,
 if the reserve of the company is equal to
$r$. Assume that $\tilde{\mathbf{C}}$ takes only two values
 $\mathbf{C}$ and $\mathbf{C}_*$ equal to
 the initial and lowered values of the premium,
 respectively, and in
addition, that the transition between these values
occurs on passing through an inert zone $(a,b)$.

The increments of the process $\xi_{a,b}(t)$ coincide with
the increments of the process
$\xi(t)$ (with intensities $\mathbf{C}$) between the last crossing of the level
 $u-a$ from below and the next crossing of the level $u-b$ from above.
The increments of $\xi_{a,b}(t)$ coincide with $\xi_*(t)$ (with intensities
$\mathbf{C}_*$) between the last crossing of the level
$u-b$ from above and the subsequent crossing of the level  $u-a$ from below.
In the notation of moment generating functions we use the symbol $*$,
which corresponds to the process  $\xi_*(t)$.  Further, if $x(t)=k$, then
\begin{multline*}
d\xi_{a,b}(t)=d\xi_k(t)I\{0\leq t\leq \tau ^{-}(u-b)\}+d\xi
_{k}^*(t)I\{\tau ^-(u-b)<t\leq \tau _b^a(u)\}+\\
+d\xi _{a,b}(t-\tau _b^a(u))I\{t>\tau _b^a(u)\},
\end{multline*}
where $\tau _b^a(u)=\inf\{t>\tau ^-(u-b):\xi _{a,b}(t)\geq u-a\}$.

\begin{figure}
\noindent\includegraphics[scale=1.3]{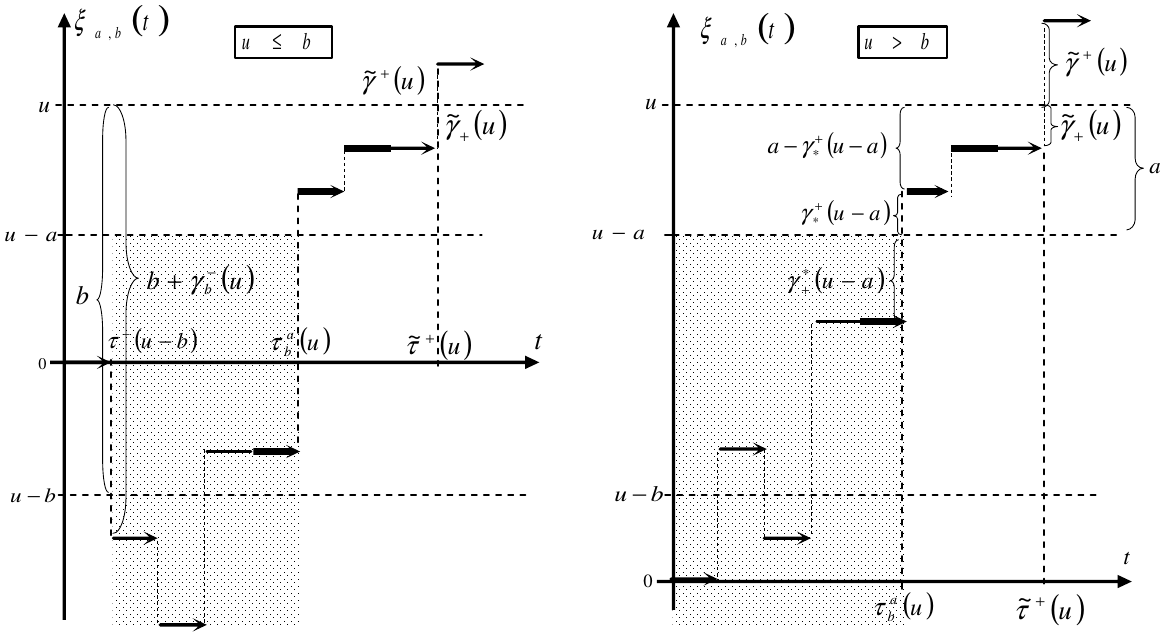}
\caption{Modified Risk Process}\label{pic1}
\end{figure}

By $\tilde{\tau}^+(u), \tilde{\gamma}_+(u)$, $
\tilde{\gamma}^+(u)$  we denote the overjump functionals for the modified
process $\xi _{a,b}(t)$ (see Figure~\ref{pic1}).  We also denote
$$
\mathbf{f}^{a,b}_s(dx,dy/u) =\mathbf{E}\left [
e^{-s\tilde{\tau}^+(u)}, \tilde{\gamma}_+(u)\in dx,
\tilde{\gamma}^+(u)\in dy, \tilde{\tau}^+(u)<\infty\right ].
$$
Then the Gerber–Shiu function can be defined as follows (see~\cite{Gerber})
$$\boldsymbol{\Phi}^{a,b} _s(u)=\int_{0}^{\infty}
\int_{0}^{\infty}w(x,y)\mathbf{f}^{a,b}_s(dx,dy/u),$$ where $w(x,y)$,
$x,y>0$  is a nonnegative function (penalty). If the parameter $s$
is regarded as the force of interest, then $\boldsymbol{\Phi}^{a,b} _s(u)$
can be regarded as a discounted expected penalty at the time to ruin.

Assume that the process $e^{u-\xi _{a,b}(t)}$ describes the price
of a stock whose variations have the form of random jumps.
 We now consider a perpetual American put option with strike price $K$.
The payoff at time $t$  is equal to $(K-e^{u-\xi_{a,b}(t)})_+$.
For the scalar case, the optimal strategy is as follows
$$\tau _\beta
=\inf\{t>0:e^{u-\xi_{a,b}(t)}<e^{\beta}\},$$ where exercise boundary
$\beta$: $e^{\beta }\leq \min(e^{u},K)$. Assume that the market is risk neutral.
 Then the price of an option is defined as the expected discounted payoff
 $$\mathbf{E}\left [  e^{-s \tau _{\beta}}(K-e^{u-\xi
_{a,b}(\tau _{\beta })})_+\right ]$$ or, in view of the fact that $\tau _{\beta
}\dot{=}\tilde{\tau}^+(u-\beta )$, as follows
$$\mathbf{E}\left [  e^{-s
\tilde{\tau}^+(u-\beta )}(K-e^{\beta -\tilde{\gamma}^+(u-\beta
)})_+\right ].$$ Therefore,
 $\boldsymbol{\Phi}^{a,b} _s(u-\beta )$ with
$w(x,y)=(K-e^{\beta-y})_+$ can also be regarded as the price of perpetual American
put option ~\cite[p.12]{GerberI}.

 \begin{thm}  For the modified process$\{\xi_{a,b}(t),x(t)\}$

\noindent{\rm{1)}} if $\, 0<u\leq b$, then
\begin{equation}\label{eq10}
\mathbf{f}^{a,b}_s(dx,dy/u)=\mathbf{f}^+_{b,s}(dx,dy/u)
+\mathbf{B}_b(s,u)\int_{0}^{\infty}\mathbf{C}e^{-\mathbf{C}z}
\mathbf{f}^{a,b}_s(dx,dy/z+b)dz;
\end{equation}
{\rm{2)}} if $\,b<u$, then
\begin{multline}\label{eq11}
\mathbf{f}^{a,b}_s(dx,dy/u)=\mathbf{f}^*_s(dx-a,dy+a/u-a)I\{x>a\}+\\
+\int_{0}^{a}
\mathbf{g}^*_s(dz/u-a)\biggl(\mathbf{f}^+_{b,s}(dx,dy/a-z)
+\mathbf{B}_b(s,a-z)\int_{0}^{\infty}\mathbf{C}e^{-\mathbf{C}v}
\mathbf{f}^{a,b}_s(dx,dy/v+b)dv\biggr),
\end{multline}
where
\begin{multline}\label{eq12}
\left(\!\mathbf{I}-\int_{0}^{\infty}\!\!\mathbf{C}e^{-\mathbf{C}z}
\int_{0}^{a}\mathbf{g}^*_s(dv/b-a+z) \mathbf{B}_b(s,a-v)dz\!\right)
\int_{0}^{\infty}\!\!\mathbf{C}e^{-\mathbf{C}z}
\mathbf{f}^{a,b}_s(dx,dy/z+b)dz=\\
=\!\!\!\int_{0}^{\infty}\!\!\!\mathbf{C}e^{-\mathbf{C}z}\!
\biggl(\!\mathbf{f}^*_s(dx-a,dy+a/b-a+z)I\{x>a\} +\!\int_{0}^{a}
\!\!\mathbf{g}^*_s(dv/b-a+z)
\mathbf{f}^+_{b,s}(dx,dy/a-v)\!\biggr)dz.
\end{multline}
\end{thm}

\begin{proof}Relation~\eqref{eq10} is an analog of the
result obtained in~\cite{Bratiychuk}.
  By using the formula of total probability and
the strict Markov property, for  $u>b$, we find
\begin{multline}\label{eq13}
 \mathbf{E}\left [  e^{-s\tilde{\tau}^+(u)},
\tilde{\gamma}_+(u)\in dx, \tilde{\gamma}^+(u)\in dy,
\tilde{\tau}^+(u)<\infty\right ]=\\
\shoveleft\quad = \mathbf{E}\biggl [ e^{-s{\tau}_{*}^+(u-a)},
{\gamma}^*_+(u-a)+a\in dx, {\gamma}_{*}^+(u-a)-a\in dy,\gamma_{*}^+(u-a)>a,
{\tau}_{*}^+(u-a)<\infty\biggr ] +\\
\shoveleft\quad \quad  +\int_{0}^{a}\mathbf{E}\left [
e^{-s{\tau}_{*}^+(u-a)}, {\gamma}_{*}^+(u-a)\in dz,
{\tau}_{*}^+(u-a)<\infty\right
]\times\\
\times\mathbf{E}\left [ e^{-s\tilde{\tau}^+(a-z)},
\tilde{\gamma}_+(a-z)\in dx, \tilde{\gamma}^+(a-z)\in dy,
\tilde{\tau}^+(a-z)<\infty\right ].
\end{multline}
This yields relation~\eqref{eq11}. Relation~\eqref{eq12}
is obtained from~\eqref{eq11} as a result of the integral transform.
\end{proof}

In the scalar case $(m=1)$ the matrix relations become
 somewhat simpler. If we set $w(x,y)=1$, then ${\Phi}^{a,b}_0(u)$
 is the ruin probability for the modified process.

 \begin{nas} For the scalar modified process $\xi_{a,b}(t)$

\noindent{\rm{1)}} if $\, 0<u\leq b$, then
\begin{equation}\label{eq14}
{\Phi}^{a,b}_0(u)=1-B_b(u)\left(1-\int_{0}^{a}g_*(dz/b-a+\theta
'_c)B_b(a-z) \right)^{-1}P^*_+(b-a+\theta '_c);
\end{equation}
{\rm{2)}} if $\,b<u$ then
\begin{multline}\label{eq15}
{\Phi}^{a,b}_0(u)=P^*_+(u-a)-\int_{0}^{a}g_*(dz/u-a)B_b(a-z)\times\\
\times \left(1-\int_{0}^{a}g_*(dz/b-a+\theta '_c)B_b(a-z)
\right)^{-1}P^*_+(b-a+\theta '_c),
\end{multline}
where
$$P^*_+(b-a+\theta '_c)=\int_{0}^{\infty}c e^{-cx}\mathrm{P}\{\xi
^+_*<b-a+x\}dx,$$
$$g_*(dz/b-a+\theta '_c)=\int_{0}^{\infty}c
e^{-cx}\mathrm{P}\{\gamma ^+_*(b-a+x)\in dz, \text{ and } \tau ^+_*(b-a+x)<\infty\}dx.$$
\end{nas}

{\bf Example.}  Assume that, for the scalar risk process, the premiums
$\xi '_n$ have exponential distributions with parameter $\tilde{c}$, whereas the claims
 $\xi _n$ obey the Erlang distribution(2):
$$\mathrm{P}\{\xi_n<x\}=\delta ^2xe^{-\delta x}, x>0.$$
It is necessary to determine the corresponding ruin probability for $u\leq b$.

According to Example 5.2~\cite{Gusak2007}, for $\mathrm{E} \xi(1)<0$ and $\mathrm{E}
\xi_*(1)<0$ we find
\begin{gather*}
\mathrm{P}\left\{\xi^+_*<u\right\}=P^*_+(u)=1-a^*_1 e^{-r^*_1 u}-a^*_2 e^{-r^*_2 u};\\
M^*_+(0+)=\frac{1}{\lambda}, dM^*_+(x)=\frac{1}{c_*|\mathrm{E}
\xi_*(1)|}dP^*_+(x), R^-(0)=p^-(0)=0,
\\  B_b(u)=\left(1-a_1 e^{-r_1 u}-a_2 e^{-r_2 u}
\right) \left(1-b_1 e^{-r_1 b}-b_2 e^{-r_2 b} \right)^{-1},
\end{gather*}
where the quantities $a_i, a^*_i$, and $b_i$ are independent of  $u$ and $b$; $r_i$ and
$r^*_i$ are positive roots of the Lundberg equation for the processes $\xi(t)$ and
$\xi _*(t)$,  respectively. By using~\eqref{eq5} and~\eqref{eq14}, we conclude
(for $u\leq b$):
\begin{equation}\label{eq16}
\mathrm{P}\left\{\tilde{\tau }^+(u)<\infty\right\}=1-\frac{(1-a_1 e^{-r_1
u}-a_2 e^{-r_2 u})(1-f^*_1 e^{-r^*_1 (b-a)}-f^*_2 e^{-r^*_2
(b-a)})}{P(u,a,b)},
\end{equation}
\begin{multline*}
P(u,a,b)=1-f_2e^{-r_1 b}-f_3e^{-r_2
b}+(g_{11}+g_{12}(b-a))e^{-\delta
(b-a)}+\\
+(g_{21}+g_{22}(b-a))e^{-\delta (b-a)-r_1
a}+(g_{31}+g_{32}(b-a))e^{-\delta (b-a)-r_2 a},
\end{multline*}
where $f_i, f^*_i$, and $g_{i,j}$ are independent of $u, a$ and $b$.

Assume that $c=1$, $c_*=4$, $\delta =20$, $\lambda _1=2$, and
$\lambda _2=1$. Then $\mathrm{E} \xi(1)=-19/10$ and $\mathrm{E}
\xi_*(1)=-2/5$. The corresponding Lundberg roots are $r_1=8$,
$r_2=95/3$ and $r^*_1=20/3$, $r^*_2=32$.

In addition, the distribution of the absolute maximum of the process $\xi _*(t)$ and
the probability of exit of the process $\xi (t)$ from the interval $(u-b,b)$
through the lower boundary are given by the formulas
$$\mathrm{P}\left\{\xi ^+_*<u\right\}=1 - \left(
        \frac{32}{57}e^{-20x/3}-\frac{9}{95}e^{-32 x} \right);$$
$$
\displaystyle B_b(u)=\mathrm{P}\left\{\xi (\tau (u,b))\leq u-b\right\}=\frac{1+\frac{49}{426}e^{-95 u/3}-
\frac{171}{355}e^{-8
u}}{1+\frac{1}{284}e^{-95b/3}-\frac{19}{355}e^{-8b}}.$$ For $a=b$,
by using~\eqref{eq16}, we arrive at the following expression for the ruin
probability
$$\mathrm{P}\left\{\tilde{\tau }^+(u)<\infty\right\}=1
-\frac{1+\frac{49}{426}e^{-95 u/3}- \frac{171}{355}e^{-8
u}}{1-\frac{45}{111328}e^{-95b/3}+\frac{19}{852}e^{-8b}}.$$

\end{document}